\documentclass[11pt,reqno]{amsart}
\usepackage{mathrsfs}
\usepackage{url}
\usepackage{mathtools}
\usepackage{latexsym,epsfig,amssymb,amsmath,amsthm,color,url,bm}
\usepackage[inline,shortlabels]{enumitem}
\usepackage{hyperref}
\usepackage[foot]{amsaddr}
\usepackage{amsmath,amsbsy}
\usepackage{mwe}
\RequirePackage[numbers]{natbib}
\usepackage{mathptmx}
\usepackage[text={16cm,24cm}]{geometry}

\allowdisplaybreaks 
 \numberwithin{equation}{section}
\newtheorem{theorem}{Theorem}[section]
\newtheorem{prop}[theorem]{Proposition}
\newtheorem{lemma}[theorem]{Lemma}

\newcommand\Item[1][]{%
  \ifx\relax#1\relax  \item \else \item[#1] \fi
  \abovedisplayskip=0pt\abovedisplayshortskip=0pt~\vspace*{-\baselineskip}}

\theoremstyle{definition}
\newtheorem{defn}[theorem]{Definition}

\theoremstyle{definition}
\newtheorem{remark}[theorem]{Remark}

\newcommand{\sr}{\textcolor{blue}}
\newcommand{\mou}{\textcolor{red}}
\newcommand{\dhr}{\textcolor{green}}

\DeclareMathOperator{\Prob}{\mathbf{P}}

\DeclareMathOperator{\Out}{Out}
\DeclareMathOperator{\E}{\widehat{E}}

\title[Ergodicity of a generalized probabilistic cellular automaton with parity-based neighbourhoods]{Ergodicity of a generalized probabilistic cellular automaton with parity-based neighbourhoods}
\date{}
\author{Dhruv Bhasin, Sayar Karmakar, Moumanti Podder, Souvik Roy}
\address{Dhruv Bhasin, Indian Institute of Science Education and Research (IISER) Pune, Dr.\ Homi Bhabha Road, Pashan, Pune 411008, Maharashtra, India.}
\address{Sayar Karmakar, University of Florida, 230 Newell Drive, Gainesville, Florida 32605, USA.}
\address{Moumanti Podder, Indian Institute of Science Education and Research (IISER) Pune, Dr.\ Homi Bhabha Road, Pashan, Pune 411008, Maharashtra, India.}
\address{Souvik Roy, Indian Statistical Institute, 203 Barrackpore Trunk Road, Kolkata 700108, West Bengal, India.}
\email{bhasin.dhruv@students.iiserpune.ac.in}
\email{sayarkarmakar@ufl.edu}
\email{moumanti@iiserpune.ac.in}
\email{souvik.2004@gmail.com}

\begin{document}
\bibliographystyle{plainnat}

\begin{abstract}

We study a one-dimensional generalized probabilistic cellular automaton $E_{p, q}$ with universe $\mathbb Z$, alphabet $\mathcal A = \{0, 1\}$, parameters $p$ and $q$ such that $0 < p+q \leq 1$ and two neighbourhoods $\mathcal N_0 = \{0, 1\}$ and $\mathcal N = \{1, 2\}$. The state $E_{p, q} \eta (x)$ of any $x \in \mathbb Z$ under the application of $E_{p, q}$ is a random variable whose probability distribution depends on the states $\eta(x + y)$ for $y \in \mathcal N_i$ where $i$ has the same parity as $x$. We establish ergodicity of this GPCA for various ranges of values of $p$ and $q$ via its connection with a suitable \textit{percolation game} on a two-dimensional lattice. For these same ranges of values of $p$ and $q$, we show that the above-mentioned game has probability $0$ of resulting in a draw.  

\end{abstract}



\maketitle
\sloppy
\section{Introduction}\label{sec:introduction}

\subsection{Overview of the paper}\label{subsec:overview} \textit{Cellular automata} (henceforth abbreviated as CAs) are an important and interesting class of objects studied in computer science and mathematics. They act on a \textit{configuration} of symbols coming from a finite \textit{alphabet} set $\mathcal A$ on the euclidean lattice $\mathbb Z^d$ for some positive integer $d$. Every vertex of $\mathbb Z^d$ is called a \textit{cell}. The space of all configurations is $\mathcal A^{\mathbb Z^d}$ and a canonical element of $\mathcal A^{\mathbb Z^d}$ is usually denoted by $\eta$. A $d-$dimensional CA is a dynamical system on configurations coming from $\mathcal A^{\mathbb Z^d}$ which is obtained by repeatedly applying a \textit{local update rule} at every cell of the lattice. The update rule is local in the sense that the way in which a given cell $\mathbf x \in \mathbb Z^d$ gets updated depends only on the symbols of some finite \textit{neighbours} of $\mathbf x$. The neighbours are determined using the following method: there is a predetermined set $\mathcal N = \{\mathbf y_1, \mathbf y_2, \dots, \mathbf y_m\} \subseteq \mathbb Z^d$ called the \textit{neighbourhood} and given any $\mathbf x \in \mathbb Z^d$, the neighbours of $\mathbf x$ are the elements of the set $\mathcal N + \mathbf x = \{\mathbf y + \mathbf x : \mathbf y \in \mathcal N\}$. Starting from a configuration $\eta_0$, we denote the configuration obtained after $t$ iterations of applying the CA by $\eta_t$.

In the case when the local update rule is random, a similar process generates a Markov chain called a \textit{probabilistic cellular automaton} (henceforth abbreviated as PCA).  A local update rule of a PCA is usually denoted by a stochastic matrix. We give a formal definition of this stochastic matrix in \ref{subsec:gpca}. Since the neighbours of a cell are translations of a given set, we can say that CAs or PCAs perform their update rules in a ``homogeneous" manner. Another way to think about this is that the the local neighbourhood of each cell in PCAs (or CAs) ``look" the same (upto translation) at every cell.

In this paper, we study a generalization of PCAs which we call \textit{generalized probabilistic cellular automata} (henceforth abbreviated as GPCAs). The main difference between PCA and GPCA is in the way the neighbours are determined. For a GPCA, for each $\mathbf x \in \mathbb Z^d$, its neighbourhood is given by $\mathcal{N}^{\mathbf{x}} = \{\mathbf{y}_{1}^{\mathbf{x}}, \mathbf{y}_{2}^{\mathbf{x}}, \ldots, \mathbf{y}_{m}^{\mathbf{x}}\} \subset \mathbb{Z}^{d}$ and this is allowed to be different for different cells $\mathbf x \in \mathbb Z^d$. Given this collection, the neighbours of a given cell $\mathbf x$ are elements of the set $\mathbf x + \mathcal N^{\mathbf x}$ which is used to determine the probability distribution of the alphabet assigned to this cell after an update of the GPCA. In a manner similar to that in PCAs, this probability distribution is determined using the local update rule. It is worth emphasizing that GPCAs are not ``homogeneous" with respect to neighbourhoods associated with various cells in $\mathbb Z^d$. This allows for the incorporation of variations in the behaviour of a GPCA from the perspective of different cells in $\mathbb Z^d$. 

The one-dimensional GPCA $E_{p, q}$ that we investigate in this paper has the underlying alphabet set $\mathcal = \{0, 1\}$ and has two sets $\mathcal N_0 = \{0, 1\}$ and $\mathcal N_1 = \{1, 2\}$ associated to it. These sets are used to determine the collection $\mathcal N^{x}$ for each $x \in \mathbb Z$ in the following way: if $x$ is even then $\mathcal N^{x} = \mathcal N_0$ and if $x$ is odd then $\mathcal N^{x} = \mathcal N_1$, and using these sets, the neighbours of each cell are decided as described earlier. The parameters $p$ and $q$ play an important role in describing our GPCA. They are members of the following set:
\begin{equation}\label{parameter_space}
(p,q) \in \mathcal{S} = \{(p', q') \in [0,1]^{2}: 0 < p'+q' \leqslant 1\}.\nonumber
\end{equation}
If the neighbours $x_1$ and $x_2$ of a given cell $x$ are both equal to $0$ then $x$ is updated to be $1$ with probability $1-p$ and $0$ with probability $p$ and otherwise, $x$ is updated to be $0$ with probability $1-q$ and $1$ with probability $q$. For a more formal definition of $E_{p, q}$, we refer the reader to \ref{subsec:gpca}.

An important question that is usually asked in the literature regarding a PCA is that of its \textit{ergodicity}. Intuitively, a PCA (GPCA) is ergodic if it, asymptotically, tends to forget its initial state $\eta_0$ and there is a unique probability distribution to which the distribution of $\eta_t$ converges as $t$ tends to infinity irrespective of the initial state. CAs are used in models of massive parallel computations (\cite{garzon2012models}), and noisy CAs can be seen as PCAs. If such a PCA is ergodic, then it means that the corresponding CA tends to forget everything about its initial state due to the noise. This is an undesirable property because we want our computational models to have tolerance against noise and not loose all the initial information.

In this paper, we investigate the question of ergodicity of our GPCA $E_{p, q}$. In \cite{buvsic2013probabilistic}, the concept of GPCAs was used (the authors called these non-homogeneous PCAs) to approximate the Majority PCA by a sequence of GPCAs and consequently deduce numerical observations about the Majority PCA. To the best of our knowledge, the question of ergodicity of any GPCA has not yet been studied in the literature. We now state our main result:

\begin{theorem}\label{thm:main_theorem_ergodic_gpca}
The GPCA $E_{p,q}$ is ergodic for following ranges of $p$ and $q$:
\begin{enumerate}
    \item $ q=0,p > p_0,$ where $p_0 \approx 0.215$ is the unique real root of $p^5 - 5p^4 + 9p^3 - 8p^2 + 6p -1$.
        \item $q > \max\left\{ \frac{-\sqrt{p^2 +2}-p+2}{2}, 0\right\}, p\geq 0$.
\end{enumerate}
\end{theorem}

The GPCA we study has a close connection with a \textit{percolation game} played on the two dimensional square lattice $\mathbb Z^2$. Let $(p, q) \in \mathcal S$. Every vertex of $\mathbb Z^2$ is, independently of others, assigned a label which reads \textit{trap} with probability $p$, \textit{target} with probability $q$ and \textit{open} with probability $r = 1 - p - q$. The two-player combinatorial game we are concerned with is played on a realization of this random assignment. A token is placed at a vertex of $\mathbb Z^2$ termed the \textit{initial vertex}. The players take turns to move the token where a move comprises moving the token from its current position $(x, y)$ to either:
\begin{enumerate}
    \item $(x, y+1)$ or $(x+1, y+1)$ if $x$ is even.
    \item $(x+1, y+1)$ or $(x+2, y+1)$ if $x$ is odd.
\end{enumerate}

A player wins the game if she is able to move the token to a vertex which has been labelled a target or if she can force her opponent to move the token to a vertex which is a trap. The game continues as long as the token stays on vertices that have been marked open. This could happen indefinitely, in which case, the game is said to result in a draw.

\subsection{Brief survey of existing literature}\label{subsec:survey} CAs are simple to define objects which have a fair amount of complexity in their properties. Computationally they are of great interest because of the fact that they can simulate any Turing machine \cite{goles1999uniform}. It is worth mentioning that the celebrated Game of Life of John Conway can be seen as a CA \cite{gardner1970fantastic}. 

As mentioned earlier, PCAs are CAs that have stochastic update rules. We refer the reader to \cite{mairesse2014around} which includes a detailed survey of PCAs where they discuss different ways PCAs can arise in combinatorics, statistical physics and theoretical computer science. Some of the primary motivations of studying PCAs are to investigate the fault tolerant capabilities of CAs (\cite{mccann2008fault, mccann2013fault}), classification of elementary CAs by using their robustness to errors as a discriminating criterion (\cite{boure2012probing, lei2021entropy} and to investigate their connections with Gibbs potentials and Gibbs measures coming from statistical physics (\cite{grinstein1985statistical, lebowitz1990statistical, rujan1987cellular}).

In \cite{holroyd2019percolation}, which is the primary motivation of our present work, the idea of percolation games was introduced. It is also a two player combinatorial game played on realizations of traps, targets and open sites on $\mathbb Z^2$. This realization is obtained in a similar manner as described in the previous section. The game studied in this paper is such that a player can move the token from $(x, y)$ to one of $(x+1, y)$ or $(x, y+1)$. The authors build a connection between this game and a PCA $A_{p, q}$. It is proved that $A_{p, q}$ is ergodic and consequently the probability of draw in the percolation game is $0$ for all values of the defining parameters $p$ and $q$. The connection  between the game and the PCA allowed the authors to use a corresponding \textit{envelope} PCA $F_{p, q}$ and the \textit{weight function} method to obtain their main result.

In \cite{bhasin2022class}, a similar connection between percolation games and PCAs was investigated. The percolation game investigated in this paper allowed a player to move the token from $(x, y)$ to one of $(x, y+2), (x+1, y+1), (x+2, y)$. The PCA investigated in this paper corresponded to these percolation games and because, in any turn, a player has three choices to move the token to, the PCA studied has neighbourhood of size $3$. By constructing a suitable weight function for this PCA, its ergodicity for any parameters $(p, q) \in \mathcal S$ was established.

In \cite{bresler2022linear}, the connections made in \cite{holroyd2019percolation} were used to come up with computer-generated weight functions for two different PCAs whose update rules can be described by $\eta_{t+1}(n) = \text{BSC}_p(\text{NAND}(\eta_t(n-1), \eta_t(n)))$ and $\eta_{t+1}(n) = \text{NAND}(\text{BSC}_p(\eta_t(n-1), \eta_t(n)))$. Here, $\text{BSC}_p$ denotes a binary symmetric channel that takes a bit as input and flips it with probability $p$ and leaves it unchanged with probability $1-p$. The weight functions obtained in this paper are for $ p \in (0, \epsilon)$ for some small $\epsilon > 0$. The authors introduce the concept of polynomial linear programming and a suitable algorithm to obtain a suitable weight function for proving the ergodicity of these PCAs.


This paper is organized as follows: in \ref{sec:model}, we describe our model--percolation game we study is described in \ref{subsec:game}, our GPCA is described in \ref{subsec:gpca}, our \textit{envelope} GPCA is introduced in \ref{subsec:envelope_gpca} and finally we draw the connection between the game and GPCA in \ref{subsec:connection} where we also state our main result. We state and prove some preliminary lemmas in \ref{sec:lemmas} and finally we proved our main theorem in \ref{sec:proof}.

\section{Description of the model}\label{sec:model}
Our model has two interrelated components, a percolation game and a generalized version of probabilistic cellular automaton. A percolation game is a combinatorial game played on a (possibly infinite) random graph and a probabilistic cellular automaton is a particular type of Markov process. We describe these components in what follows.  

\subsection{Description of our game}\label{subsec:game} We consider a percolation game on $\mathbb Z^2$. To begin with, every vertex of $\mathbb Z^2$, independent of other vertices, is labelled a trap with probability $p$, a target with probability $q$ and it is open with probability $1-p-q$. Here $p$ and $q$ are the parameters of the game and $(p,q) \in \mathcal S$ where $\mathcal S$.

The game is played between two players who take turns to move the token. Initially, a token is kept at the origin. 
If the token is at vertex $(x, y)$ then the player, in turn, can move it to:
\begin{enumerate}
    \item $(x,y+1)$ or $(x+1,y+1)$ if $x$ is even.
    \item $(x+1, y+1)$ or $(x+2, y+1)$ if $x$ is odd.
\end{enumerate}

We demonstrate the directed graph in which the game is being played in Figure~\ref{gameplay}. 

\begin{figure}
        \centering
        \includegraphics[width=0.5\linewidth]{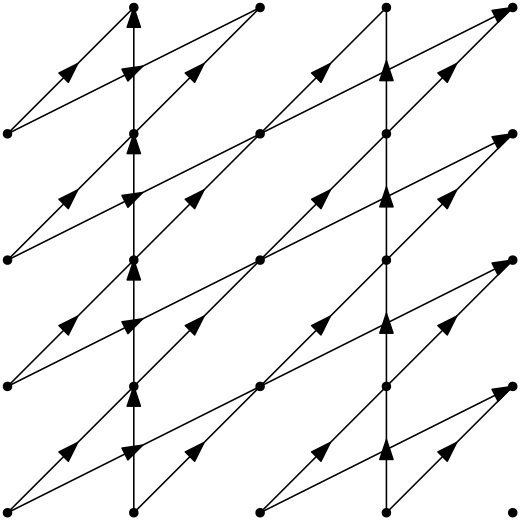}
        \caption{The graph on which the game is played}
        \label{gameplay}
    \end{figure}
    
A player wins the game if she is able to move her token to a target or if the other player is forced to move the token to a trap. The game continues as long as the token remains on an open node. If the game continues forever it is said to result in a draw. We note here that both the players decide their strategies after the graph has been assigned with the trap/target/open labelling and that they have access to the whole information of the graph and form their strategies accordingly.

\subsection{Description of our GPCA}\label{subsec:gpca} In the existing literature for (deterministic) cellular automata (henceforth abbreviated as CA), and extending on it, for probabilistic cellular automata (henceforth abbreviated as PCA), we usually come across the following three principal components:   
\begin{enumerate*}
\item a finite set $\mathcal{A}$ of symbols termed the \emph{alphabet},
\item a finite set of indices $\mathcal{N} = \{\mathbf{y}_{1}, \mathbf{y}_{2}, \ldots, \mathbf{y}_{m}\} \subset \mathbb{Z}^{d}$ (where $d$ indicates the dimension of the CA or PCA under consideration) which is called the \emph{neighbourhood},
\item and local update rule $\varphi: \mathcal{A}^{m} \times \mathcal{A} \rightarrow [0,1]$ that is deterministic in case of CAs and stochastic in case of PCAs.
\end{enumerate*}
In addition to these essential ingredients, for a $d$-dimensional CA or a PCA, the \emph{state space} $\Omega$ equals $\mathcal{A}^{\mathbb{Z}^{d}}$, and the elements $\eta = (\eta(\mathbf{x}): \mathbf{x} \in \mathbb{Z}^{d})$ of $\Omega$ are termed \emph{configurations}. We would like to draw the reader's attention to the fact that the concept of a neighbourhood is defined in a uniform way throughout the lattice in the sense that the neighbourhood $\mathcal N$ is independent of the choice of $\mathbf{x} \in \mathbb Z^d$. To generalise the notion of a PCA to that of a generalised probabilistic cellular automaton (henceforth GPCA), we make a change to the notion of neighbourhood $\mathcal N$ and make them dependent on vertices. Thus, the neighbourhood for a vertex $x$ is: $\mathcal{N}^{\mathbf{x}} = \{\mathbf{y}_{1}^{\mathbf{x}}, \mathbf{y}_{2}^{\mathbf{x}}, \ldots, \mathbf{y}_{m}^{\mathbf{x}}\} \subset \mathbb{Z}^{d}$,  and the neighbours of a vertex $\mathbf{x}$ are the elements of the set $\mathbf{x}+\mathcal N^{\mathbf{x}}$. It should be noticed that, though the set of indices may vary over  nodes, the number of indices remain constant. This enables us to use a common local update rule for all nodes. More formally, the update rules for a GPCA $F$ are represented by a stochastic matrix $\varphi: \mathcal{A}^{m} \times \mathcal{A} \rightarrow [0,1]$ defined as follows:
\begin{align}\label{general_update_rule_eq}
\Prob[F \eta(\mathbf{x}) = b\big|\eta(\mathbf{x}+\mathbf{y}_{i}^{\mathbf{x}}) = a_{i} \text{ for all } 1 \leqslant i \leqslant m] = \varphi(a_{1}, a_{2}, \ldots, a_{m}, b) \text{ for all } b \in \mathcal{A},
\end{align}   
where, by definition of stochastic matrices, we have $\varphi(a_{1}, a_{2}, \ldots, a_{m}, b) \geqslant 0$ and $\sum_{b \in \mathcal{A}} \varphi(a_{1}, a_{2}, \ldots, a_{m}, b) = 1$, for all $a_{1}, a_{2}, \ldots, a_{m}, b \in \mathcal{A}$. 

We now define the notion of ergodicity for a $d-$dimensional GPCA $F$. The definition is analogous to the usual definition for PCAs. Let $\mathcal F$ denote the $\sigma-$field that is generated by the cylinder sets of $\Omega  = \mathcal A^{{\mathbb Z}^d}$, and let $\mathbb D$ denote the set of all probability measures over $\Omega$ and defined with respect to the sigma field $\mathcal F$. We put $F^1\eta = F\eta$ and for any natural number $t\geq 2$, we define $F^t\eta = F(F^{t-1}\eta)$ for $\eta \in \Omega$. This definition extend naturally to random $\eta$ following the distribution $\mu \in \mathbb D$. and we let $F^t\mu$ (simply written $F\mu$ when $t=1$ denote the probability distribution of the (also random) configuration $F^t\mu$.

We say that $\mu$ is a \textit{stationary} measure for a GPCA $F$ if $F\mu = \mu$. A GPCA $F$ is said to be \textit{ergodic} if it has unique stationary measure, such that for every probability measure $\nu$ on $\Omega$, the sequence $F^t\nu$ converges weakly to $\mu$ as $t \rightarrow \infty$. That is, the GPCA forgets its initial state and converges to a unique measure.

We now introduce and explain the one-dimensional GPCA $E_{p,q}$ that we study in this paper. The definition itself is the first among quite a few aspects in which the GPCA $E_{p,q}$ differs from more commonly studied PCAs in the literature. While $E_{p,q}$, as usual, is endowed with the alphabet $\mathcal{A} = \{0,1\}$ and stochastic local update rules captured by the stochastic matrix $\varphi_{p,q}$ defined below, we now need \emph{two different} finite subsets of $\mathbb{Z}$, denoted $\mathcal{N}_{1}$ and $\mathcal{N}_{0}$, to serve as neighbourhoods in the definition of this GPCA. For $n \in \mathbb{Z}$, the neighbourhood we consider is $\mathcal{N}_{1}$ if $n$ is odd, and $\mathcal{N}_{0}$ if $n$ is even. 

To elucidate, we set $\mathcal{N}_{1} = \{1,2\}$ and $\mathcal{N}_{0} = \{0,1\}$. Given a configuration $\eta = (\eta(n): n \in \mathbb{Z}) \in \Omega$, where $\Omega = \mathcal{A}^{\mathbb{Z}}$ is the state space for the GPCA $E_{p,q}$, the state $E_{p,q}\eta(n)$ of the site $n$ under the application of $E_{p,q}$ is a random variable that is a function of $(\eta(n+i): i \in \mathcal{N}_{1}\}$, i.e.\ of the states $\eta(n+1)$ and $\eta(n+2)$ of sites $n+1$ and $n+2$, when $n$ is odd. On the other hand, the random variable $E_{p,q}\eta(n)$ is a function of $(\eta(n+i): i \in \mathcal{N}_{0}\}$, i.e.\ of the states $\eta(n)$ and $\eta(n+1)$ of sites $n$ and $n+1$, when $n$ is even. In particular both the random variables $E_{p, q}\eta(2n-1)$ and $E_{p,q}\eta(2n)$ are equidistributed since they are functions of the states $\eta(2n)$ and $\eta(2n+1)$.

For the GPCA that we are concerned with in this paper, stochastic update rule $\varphi_{p,q}$ is defined as follows: 
\begin{equation}\label{PCA_rule_1}
 \varphi_{p,q}(0, 0, b) =
  \begin{cases} 
   p & \text{if } b = 0, \\
   1-p & \text{if } b = 1,
  \end{cases}
\end{equation}
and
\begin{equation}\label{PCA_rule_2}
 \varphi_{p,q}(a_{0}, a_{1}, b) =
  \begin{cases} 
   1-q & \text{if } b = 0, \\
   q & \text{if } b = 1,
  \end{cases} \quad \text{for all } (a_{0}, a_{1}) \in \mathcal{A}^{2} \setminus \{(0,0)\}.
\end{equation}

Note that even though the definition of the stochastic rule is independent of the individual vertices, following \ref{general_update_rule_eq}, the way the individual vertices use this stochastic rule varies from vertex to vertex, and in our GPCA, $E_{p,q}$, the usage of this rule depends on the parity of the vertex.

\subsection{Description of our envelope GPCA}\label{subsec:envelope_gpca} 
The envelope $\E_{p,q}$ to the GPCA $E_{p,q}$ is obtained by first extending the alphabet $\mathcal{A}$ to $\hat{\mathcal{A}} = \{0,?,1\}$, and then introducing the local update rules via the stochastic matrix $\hat{\varphi}_{p,q}$ defined as follows:
\begin{equation}\label{envelope_PCA_rule_1}
 \widehat{\varphi}_{p,q}(0, 0, b) =
  \begin{cases} 
   p & \text{if } b = 0, \\
   1-p & \text{if } b = 1,
  \end{cases}
\end{equation}
\begin{equation}\label{envelope_PCA_rule_2}
 \widehat{\varphi}_{p,q}(a_{0}, a_{1}, b) =
  \begin{cases} 
   1-q & \text{if } b = 0, \\
   q & \text{if } b = 1,
  \end{cases} \quad \text{for all } (a_{0}, a_{1}) \in \hat{\mathcal{A}}^{2} \setminus \{0,?\}^{2}
\end{equation}
and
\begin{equation}\label{envelope_PCA_rule_3}
 \widehat{\varphi}_{p,q}(a_{0}, a_{1}, b) =
  \begin{cases} 
   p & \text{if } b = 0, \\
   q & \text{if } b = 1,\\
   r = 1-p-q & \text{if } b = ?,
  \end{cases} \quad \text{for all } (a_{0}, a_{1}) \in \{0,?\}^{2} \setminus (0,0).
\end{equation}
The neighbourhoods $\mathcal{N}_{1}$ and $\mathcal{N}_{0}$ corresponding to odd and even vertices remain the same as in $E_{p,q}$. We reiterate here that this means that for an odd $n\in\mathbb Z$, $\E_{p,q}\eta(n)$ depends on $\eta(n+1)$ and $\eta(n+2)$ whereas for even $n\in\mathbb Z$, $\E_{p,q}\eta(n)$ depends on $\eta(n)$ and $\eta(n+1)$. In particular, $\E_{p,q}\eta(2n-1), \E_{p, q}\eta(2n)$ both depend on $\eta(2n)$ and $\eta(2n+1)$.

\subsection{Connection between our game and the envelope GPCA}\label{subsec:connection} In this section, we delve into the underlying connection between the game we play and our GPCA. As stated in the previous section, recall that we play our game on $\mathbb Z^2$. Suppose every vertex of $\mathbb Z^2$ has been assigned to be a trap/target/open with corresponding probabilities. We define sets of vertices $W$, $L$, and $D$ which we will need to establish this connection. We say that a vertex $(x, y)$ is in $W$ if the player who plays the first round wins the game when the token is placed at $(x, y)$ at the beginning of the game. Similarly, we say that a vertex $(x, y)$ is in $L$ if the person who plays first, loses the game when the token is placed at $(x,y)$ at the start of the game. Finally $D$ is the set of vertices $(x,y)$ such that when the token is initially kept at $(x, y)$, the game is results in a draw. We follow the convention that if the vertex is a trap, then it is assigned a label $W$ and if it is a target, it is assigned a label $L$. 

For $k \in \mathbb Z$, let $H_k = \{(x, k)\in \mathbb Z^2: x\in \mathbb Z\}$ be the horizontal line at height $k$. We claim that if, for any arbitrary realization of the graph into traps/targets/open sites, we know the labelling of vertices (into $W, L$ and $D$) in the horizontal line $H_{k+1}$, then, using this information and the pre-assigned trap/target/open labelling, we can deduce the labelling of vertices in the horizontal line $H_k$. We note that since the neighbours of $(2x-1, y)$ and $(2x, y)$ are the same, it suffices to give arguments for $(2x, y)$.:

\begin{enumerate}
    \item Suppose both of the vertices $(2x, y+1)$ and $(2x+1, y+1)$ are labelled $W$, then no matter what move the first player makes, the second player is guaranteed to win the game. Hence, the vertex $(2x, y)$ should be labelled $L$ if it is open. Otherwise, if it is a trap, it gets label $W$, which happens with probability $p$.
    \item If at least one of the vertices $(2x, y+1)$ and $(2x+1, y+1)$ is labelled $L$ then the first player can always move to this vertex in order to ensure victory in the game. This means that $(2x, y)$ should be labelled $W$ is it is open. Otherwise, if it is a target, it gets label $L$, which happens with probability $q$.
    \item If none of the vertices $(2x, y+1)$ and $(2x+1, y+1)$ is labelled $L$ but at least one of them is labelled $D$, then the first player should move to this vertex to ensure a draw as the best case scenario. Hence, the vertex $(2x, y)$ should be labelled as $D$ if it is open. Otherwise, it gets labelled $W$ if it is a trap, which happens with probability $p$ and it gets label $L$ if it is a target which happens with probability $q$.
\end{enumerate}

We identify the symbols $0$ with the label $W$, $1$ with $L$ and $?$ with $D$. With this assignment of symbols in hand, it is clear that if $\eta \in \{0, 1, ?\}^{\mathbb Z}$ is a configuration of symbols on $H_{k+1}$ then the random configuration on $H_k$ is given by $\E_{p, q} \eta$.  Finally, the connection between the GPCA and the game we study is established via the following results.

\begin{prop}\label{prop:ergodic_env}
The GPCA $E_{p,q}$ is ergodic if and only if the corresponding GPCA $\widehat{E}_{p,q}$ if ergodic.
\end{prop}

\begin{prop}\label{prop:game_ergodic}
For every $(p,q) \in \mathcal{S}$, our game has probability $0$ of ending in a draw if and only if the GPCA $E_{p,q}$ is ergodic.
\end{prop}

Before stating the next theorem, we would like to delve on its importance and connections with the above stated propositions. We will show in the next section that for $\E_{p, q}$ to be ergodic, it suffices to show that it has no stationary distribution which assigns positive probability to the occurrence of $?$. Once we have shown that the GPCA is ergodic, using the above propositions, it follows that our game has $0$ probability of draw. From Proposition~\ref{prop:ergodic_env} and Proposition~\ref{prop:game_ergodic}, it is easy to see that Theorem~\ref{thm:main_theorem_ergodic_gpca} and Theorem~\ref{thm:main_theorem} are equivalent. In order to establish Theorem~\ref{thm:main_theorem_ergodic_gpca}, we establish Theorem~\ref{thm:main_theorem}.

\begin{theorem}\label{thm:main_theorem}
The envelope GPCA $\widehat{E}_{p,q}$ admits no stationary distribution $\mu$ that assigns positive probability to the occurrence of $?$, for following ranges of $p$ and $q$:
\begin{enumerate}
    \item $ q=0, p > p_0,$ where $p_0 \approx 0.215$ is the unique real root of $p^5 - 5p^4 + 9p^3 - 8p^2 + 6p -1$.
        \item $q > \max\left\{ \frac{-\sqrt{p^2 +2}-p+2}{2}, 0\right\}, p\geq 0$.
\end{enumerate}

\end{theorem}

\begin{remark}
We note that when $ p = 0$, Theorem~\ref{thm:main_theorem} tells us that whenever $q > \frac{2 - \sqrt{2}}{2} \approx 0.2929$ the GPCA $\E_{p, q}$ admits no stationary distribution which assigns a positive probability to the occurrence of $?$. The region covered by Theorem~\ref{thm:main_theorem}(ii) is shown in the figure \ref{fig:region_thm}.

\begin{figure}\label{fig:region_thm}
        \centering
        \includegraphics[width=0.65\linewidth]{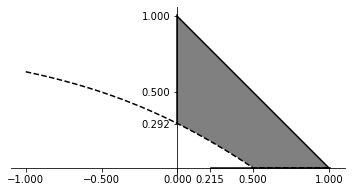}
        \caption{Region covered by Theorem~\ref{thm:main_theorem} (dotted curve is not included in the region)}
    \end{figure}

\end{remark}

\begin{remark}
Note that the polynomial $p^5 - 5p^4 + 9p^3 - 8p^2 + 6p -1$ assumes positive values for $p>p_0$. This polynomial, along with some other polynomials shows up in \ref{coeff_poly} where in order to obtain our result, we want all those polynomials to take positive values. It turns out that all the polynomials are positive in $ ( p_0, p_1) $ where $p_1\approx 0.555$ and is the second largest root of $p^3-2p^2-p+1$.  
\end{remark}


\section{Some preliminary lemmas}\label{sec:lemmas}

\begin{lemma}\label{lemma_basic}
Let $\mu$ and $\nu$ be two probability distributions on $\{0, ? , 1\}^{\mathbb Z}$.
\begin{enumerate}
    \item\label{lemma:basic_1} If $\mu \leq \nu$, where $\leq$ is stochastic domination with respect to the coordinate-wise partial order induced by $0 < ? < 1$, then $E_{p,q}\mu \geq E_{p,q}\nu$.
    \item\label{lemma:basic_2} If $\mu \trianglelefteq \nu$, where $\trianglelefteq$ denotes stochastic domination with respect to the coordinate-wise partial order induced by $0 \triangleleft ? \triangleright 1$, then $E_{p,q}\mu \trianglelefteq E_{p,q}\nu$.
\end{enumerate}

\end{lemma}
\begin{proof}
Given a a configuration $(\eta(\mathbf x) \in \{0, 1, ?\}: \mathbf x \in H_{k+1})$ and a configuration $\omega$ of traps, targets and open sites in $H_k$, we define $\E_{\omega}\eta$ to be the deterministic configuration of symbols $\{0, ?, 1\}$ coming from the relations described in section~\ref{subsec:connection}. We first fix two configurations $( \eta(\mathbf x) : \mathbf x \in H_{k+1})$ and ($\tilde{\eta}(\mathbf x) : \mathbf x \in H_{k+1} )$, in such a way that $\eta(\mathbf x) \leq \tilde{\eta}(\mathbf x)$ for all $\mathbf x \in H_{k+1}$ where $\leq$ is the coordinate-wise partial order on configurations induced by $0 < ? < 1$. Let us fix a configuration $\omega$ of traps, targets and open sites in $H_k$. If $\mathbf x \in H_k$ is not open then by our convention $\E_{\omega}\eta(\mathbf x) = \E_{\omega}\tilde{\eta}(\mathbf x)$. Suppose $\mathbf x\in H_{k}$ is open and let $\mathbf x = (x, k)$. We assume, without loss of generality, that $x$ is even (the same set of arguments apply in the case $x$ is odd as $(x, k)$ and $(x+1, k)$ have the same set of neighbours). We consider the following cases:
\begin{enumerate}
    \item Suppose $\tilde{\eta}(x, k+1) = \tilde{\eta}(x+1, k+1) = 0$. Since $\eta(\mathbf x) \leq \widehat{\eta}(\mathbf x)$ for all $\mathbf x \in H_{k+1}$, it follows that $\eta(x, k+1) = \eta(x+1, k+1) = 0$. By definition of $\E_{\omega}$, this implies that $\E_{\omega}\tilde{\eta}(\mathbf x) = \E_{\omega}\eta(\mathbf x) = 1$.
    \item Suppose $(\tilde{\eta}(x, k+1), \tilde{\eta}(x+1, k+1)) \in \{0, ?\}^2 \setminus \{(0, 0)\}$. By the means of the fact that $\eta(\mathbf x) \leq \widehat{\eta}(\mathbf x)$ for all $\mathbf x \in H_{k+1}$, it follows that $(\eta(x, k+1), \eta(x+1, k+1)) \in \{0, ?\}^2$. By the definition of $\E_{\omega}$ $\E_{\omega}\eta(\mathbf x) \in \{1, ?\}$ and $\E_{\omega}\tilde{\eta}(\mathbf x) = ?$. This allows us to conclude that $\E_{\omega}\eta(\mathbf x) \geq ? = \E_{\omega}\tilde{\eta}(\mathbf x)$.
    \item Suppose $(\tilde{\eta}(x, k+1), \tilde{\eta}(x+1, k+1)) \in \{1, 0, ?\}^2 \setminus \{0, ?\}^2$. Using the definition of $\E_{\omega}$, we have that $\E_{\omega}\tilde{\eta}(\mathbf x) = 0$. We can conclude that, no matter what $(\eta(x, k+1), \eta(x+1, k+1))$ is, we must have $\E_{\omega}\eta(\mathbf x) \geq 0 = \E_{\omega}\tilde{\eta}(\mathbf x)$.
\end{enumerate}
For the second part of the lemma, we begin with $\eta(\mathbf x) \trianglelefteq \tilde{\eta}(\mathbf x)$ for all $\mathbf x \in H_{k+1}$. We again fix a trap/target/open labelling, say $\omega$, on $H_k$ and whenever $(x, k) = \mathbf x \in H_k$ is open, we consider the following cases (assuming as earlier that $x$ is even):
\begin{enumerate}
    \item Suppose $\eta(x, k+1) = \eta(x+1, k+1)) = 0$. By virtue of the fact that $\eta(\mathbf x) \trianglelefteq \tilde{\eta}(\mathbf x)$ for all $\mathbf x \in H_{k+1}$, we have that  $(\tilde{\eta}(x, k+1), \tilde{\eta}(x+1, k+1)) \in \{0, ?\}^2$. By the definition of $\E_{\omega}$, it follows that $\E_{\omega}\tilde{\eta}(\mathbf x) \in \{1, ?\}$ and that $\E_{\omega}\eta(\mathbf x) = 1$. From this, we can conclude that $\E_{\omega}\eta(\mathbf x) = 1 \trianglelefteq \E_{\omega}\tilde{\eta}(\mathbf x)$.
    \item Suppose $(\eta(x, k+1), \eta(x+1, k+1)) \in \{0, ?\}^2 \setminus \{(0, 0)\} $. By means of the fact that $\eta(\mathbf x) \trianglelefteq \tilde{\eta}(\mathbf x)$ for all $\mathbf x \in H_{k+1}$, it follows that $(\tilde{\eta}(x, k+1), \tilde{\eta}(x+1, k+1)) \in \{0, ?\}^2 \setminus \{(0, 0)\}$. By the definition of $\E_{\omega}$, we have that $\E_{\omega}\eta(\mathbf x) = ? = \E_{\omega}\tilde{\eta}(\mathbf x)$. The last equation satisfies the required inequality. 
    \item Suppose $(\eta(x, k+1), \eta(x+1, k+1)) \in \{1, 0, ?\}^2 \setminus \{0, ?\}^2$. Since $\eta(\mathbf x) \trianglelefteq \tilde{\eta}(\mathbf x)$ for all $\mathbf x \in H_{k+1}$, we must have that $(\tilde{\eta}(x, k+1), \tilde{\eta}(x+1, k+1)) \in \{1, 0, ?\}^2\setminus \{(0, 0)\}$. By the defintion of $\E_{\omega}$, we have $\E_{\omega}\tilde{\eta}(\mathbf x) \in \{0, ?\}$ and $\E_{\omega}\eta(\mathbf x) = 0$. This allows us to conclude that $\E_{\omega}\eta(\mathbf x) = 0 \trianglelefteq \E_{\omega}\eta(\mathbf x)$.
\end{enumerate}
This completes the proof.
\end{proof}

\emph{Proof of Proposition \ref{prop:ergodic_env}} If $\E_{p, q}$ is ergodic then clearly $E_{p, q}$ is ergodic. Suppose $E_{p, q}$ is ergodic. Let $\bar{0} \in \{0, 1, ?\}^{\mathbb Z}$ be the configuration which is identically $0$ everywhere and similarly let $\bar{1}$ be the configuration which is $1$ everywhere.  Let $\delta_0$ and $\delta_1$ be the distributions which assign probability $1$ to the configurations $\bar{0}$ and $\bar{1}$ respectively. For any distribution $\mu$, we have $\delta_0 \leq \mu \leq \delta_1$ where $\leq$ indicates stochastic domination with respect to the coordinate-wise partial order induced by $0 < ? < 1$ (as in Lemma~\ref{lemma_basic}). By Lemma~\ref{lemma_basic}, we conclude that 
\begin{equation}\label{eq:lemma_basic_1}
    \E_{p, q}^{2k} \delta_0 \leq \E_{p, q}^{2k} \mu \leq \E_{p, q}^{2k} \delta_1
\end{equation}
and that 
\begin{equation}\label{eq:lemma_basic_2}
    \E_{p, q}^{2k+1} \delta_0 \geq \E_{p, q}^{2k+1} \mu \geq \E_{p, q}^{2k+1} \delta_1
\end{equation} 
for each $k \in \mathbb N_0$. It is straightforward to see that $\E_{p, q}$ when restricted to the sub-alphabet $\{0, 1\}$ boils down to the original GPCA $E_{p, q}$ yielding the equalities $\E_{p, q}^k \delta_0 = E_{p, q} \delta_0$ and $\E_{p, q}^k \delta_1 = E_{p, q} \delta_1$. This allows us to rewrite equation~\ref{eq:lemma_basic_1} and ~\ref{eq:lemma_basic_2} as follows:
\begin{equation}\label{eq:lemma_basic_3}
    E_{p, q}^{2k} \delta_0 \leq \E_{p, q}^{2k} \mu \leq E_{p, q}^{2k} \delta_1
\end{equation}
\begin{equation}\label{eq:lemma_basic_4}
    E_{p, q}^{2k+1} \delta_0 \geq \E_{p, q}^{2k+1} \mu \geq E_{p, q}^{2k+1} \delta_1
\end{equation} 
Since $E_{p, q}$ is ergodic,  we know that it has a unique stationary distribution $\nu$ and each of $\E_{p, q}^k\mu$ and $\E_{p, q}^k\mu$ converge to $\nu$ as $k$ approaches infinity. Therefore, ~\ref{eq:lemma_basic_3} and \ref{eq:lemma_basic_4} together imply that $\E_{p, q} \mu $ also converges to $\nu$. This establishes that $\E_{p, q}$ is also ergodic. 

\emph{Proof of Proposition \ref{prop:game_ergodic}} The proof of this proposition follows via an argument identical to the proof of Proposition 2.2 of \cite{holroyd2019percolation}.



\section{Proof of Theorem~\ref{thm:main_theorem}}\label{sec:proof}

For some $i\in \mathbb Z$ and $k \in \mathbb N$, let $S = \{i, \dots, i+k-1\} \subseteq \mathbb Z$ be a finite index set of size $k$ and $a_1, a_2, \dots a _k \in \{0, 1, ?\}$ be some symbols. We define $( a_1, a_2, \dots , a_k)_i = \{ \eta \in \{0, 1 ,?\}^{\mathbb Z} : \eta(j) = a_{j - i + 1} \hspace{1.75 mm} \forall j \in S\} $ as a cylinder set of length $k$ indexed at $i$. 

The lemma that we state now follows from the definition of the GPCA $\E_{p, q}$ (in particular we use equation~\ref{general_update_rule_eq}, the neighbourhoods of $\E_{p, q}$ and the fact that the state of each site is updated independently under the action of $\E_{p, q}$.

\begin{lemma}\label{lemma:e_pq}
Let $\mu$ be a given distribution on $\{0, 1, ?\}^{\mathbb Z}$ and $\widehat{varphi}_{p,q}$ be the local updation rule of the GPCA $E_{p, q}$ as describe in section~\ref{subsec:envelope_gpca}. Then the following holds:
\begin{enumerate}
    \item If $i$ is odd and $(b_1, \dots , b_{2k}) \in \{0, 1, ?\}^{2k}$, then
    \begin{equation}
        \E_{p,q} \mu (b_1, \dots , b_{2k})_i = \sum_{(a_1, \dots , a_{2k}) \in \{0, 1, ?\}^{2k} } f_1(a_1, \dots , a_{2k}, b_1, \dots , b_{2k}) \mu(a_1, \dots , a_{2k})_{i+1} 
    \end{equation}
    where 
    \begin{align*}
    f_1(a_1, \dots , a_{2k}, b_1, \dots , b_{2k}) = \prod_{t = 1}^k \widehat{\varphi}_{p,q}(a_{2t - 1}, a_{2t}, b_{2t-1}) \widehat{\varphi}_{p,q}(a_{2t-1}, a_{2t}, b_{2t}).
    \end{align*}
    \item If $i$ is odd and $(b_1, \dots , b_{2k-1}) \in \{0, 1, ?\}^{2k-1}$, then
    \begin{equation}
        \E_{p,q} \mu (b_1, \dots , b_{2k-1})_i = \sum_{(a_1 \dots a_{2k}) \in \{0, 1, ?\}^{2k} } f_2(a_1 \dots a_{2k}, b_1, \dots , b_{2k-1}) \mu(a_1 \dots a_{2k})_{i+1} 
    \end{equation}
    where 
    \begin{align*}
        f_2(a_1 \dots a_{2k}, b_1, \dots , b_{2k-1}) = (\prod_{t = 1}^{k-1} \widehat{\varphi}_{p,q}(a_{2t - 1}, a_{2t}, b_{2t-1}) \widehat{\varphi}_{p,q}(a_{2t-1}, a_{2t}, b_{2t}))\widehat{\varphi}_{p,q}(a_{2k-1}, a_{2k}, b_{2k-1})
    \end{align*}
    \item If $i$ is even and $(b_1, \dots , b_{2k}) \in \{0, 1, ?\}^{2k}$, then 
    \begin{equation}
        \E_{p,q} \mu (b_1, \dots , b_{2k})_i = \sum_{(a_1, \dots, a_{2k+2}) \in \{0, 1, ?\}^{2k+2} } f_3(a_1, \dots, a_{2k+2}, b_1, \dots , b_{2k}) \mu(a_1, \dots, a_{2k+2})_{i} 
    \end{equation}
    where 
    \begin{align*}
        f_3(a_1, \dots, a_{2k+2}, b_1, \dots &, b_{2k}) = \\
        &\widehat{\varphi}_{p,q}(a_{1}, a_{2}, b_{1})(\prod_{t = 1}^k \widehat{\varphi}_{p,q}(a_{2t - 1}, a_{2t}, b_{2t-1}) \widehat{\varphi}_{p,q}(a_{2t-1}, a_{2t}, b_{2t})) \widehat{\varphi}_{p,q}(a_{2k+1}, a_{2k}, b_{2k})
    \end{align*}
    \item If $i$ is even and $(b_1, \dots , b_{2k-1}) \in \{0, 1, ?\}^{2k-1}$, then 
    \begin{equation}
        \E_{p,q} \mu (b_1, \dots , b_{2k-1})_i = \sum_{(a_1, \dots, a_{2k}) \in \{0, 1, ?\}^{2k} } f_4(a_1, \dots, a_{2k}, b_1, \dots , b_{2k-1}) \mu(a_1, \dots, a_{2k})_{i} 
    \end{equation}
    where 
    \begin{align*}
        f_4(a_1, \dots, a_{2k}, b_1, \dots , b_{2k-1}) = \widehat{\varphi}_{p,q}(a_{1}, a_{2}, b_{1})(\prod_{t = 1}^{k-1} \widehat{\varphi}_{p,q}(a_{2t - 1}, a_{2t}, b_{2t-1}) \widehat{\varphi}_{p,q}(a_{2t-1}, a_{2t}, b_{2t}))
    \end{align*}
\end{enumerate}
\end{lemma}

With this lemma in hand, we show that it is enough to prove Theorem~\ref{thm:main_theorem} for distributions $\mu$ on $\{0, 1, ?\}^{\mathbb Z}$ satisfying:
\begin{enumerate}
\item $\mu(b_1b_2\dots b_k)_i = \mu(b_1b_2\dots b_k)_{i+2} \forall i \in \mathbb Z$.
    \item $\mu(b_1b_2\dots b_{2k})_i = \mu(b_{2k}b_{2k-1}\dots b_1)_i \forall i \in \mathbb Z$.
    \item $\mu(b_1b_2\dots b_{2k+1})_i = \mu(b_{2k+1}b_{2k}\dots b_1)_{i+1} \forall i \in \mathbb Z$.
    \item If $i+j$ is even then, $\mu(b_1b_2\dots b_j b_{j+1} \dots b_k)_i = \mu(b_1 b_2 \dots b_{j-1} b_{j+1} b_j b_{j+2} \dots b_k)_i \forall i \in \mathbb Z$.
    
\end{enumerate}

\begin{lemma}\label{lemma:suffice}
Let $\mu$ be a distribution on $\{0, 1, ?\}^{\mathbb Z}$ satisfying conditions (i)-(iv) above. Then $\E_{p , q} \mu$ also satisfies the conditions (i)-(iv).
\end{lemma}
\begin{proof} We prove each of the properties separately in the following parts:
\begin{enumerate}
    \item In this part, we show that $E_{p,q} \mu$ satisfies the first property. We begin with a cylinder set $(b_1 b_2 \dots b_{2k})_i$ indexed at an odd integer $i$. The first part of Lemma~\ref{lemma:e_pq} gives us:  

\begin{equation}\label{eq:suff_1}
        \E_{p,q} \mu (b_1, \dots , b_{2k})_i = \sum_{(a_1, \dots , a_{2k}) \in \{0, 1, ?\}^{2k} } f_1(a_1, \dots , a_{2k}, b_1, \dots , b_{2k}) \mu(a_1, \dots , a_{2k})_{i+1} 
    \end{equation}
    where
    \begin{align*}
    f_1(a_1, \dots , a_{2k}, b_1, \dots , b_{2k}) = \prod_{t = 1}^k \widehat{\varphi}_{p,q}(a_{2t - 1}, a_{2t}, b_{2t-1}) \widehat{\varphi}_{p,q}(a_{2t-1}, a_{2t}, b_{2t}).
    \end{align*}

Since $\mu$ satisfies the condition (i) above, we have that $\mu(a_1,\dots , a_{2k})_{i+1} = \mu(a_1, \dots , a_{2k})_{i+3}$. This allows us to rewrite equation~\ref{eq:suff_1} as

\begin{align*}
        \E_{p,q} \mu (b_1, \dots , b_{2k})_i &= \sum_{(a_1,\dots , a_{2k}) \in \{0, 1, ?\}^{2k} } f_1(a_1,\dots , a_{2k}, b_1, \dots , b_{2k}) \mu(a_1,\dots , a_{2k})_{i+1}  \\
         &= \sum_{(a_1,\dots , a_{2k}) \in \{0, 1, ?\}^{2k} } f_1(a_1,\dots , a_{2k}, b_1, \dots , b_{2k}) \mu(a_1,\dots , a_{2k})_{i+3}  \\
        &= \E_{p,q} \mu (b_1, \dots , b_{2k})_{i+2}
    \end{align*}
where in the last step, we use the first part of Lemma~\ref{lemma:e_pq}. Proofs of the other cases concerning the parities of $i$ and $k$ can be dealt with in a similar way.
\item In this part, we prove that $E_{p,q} \mu$ satisfies the second property. For this we begin with a cylinder set $(b_1, \dots, b_{2k})_i $ for an odd $i$ and use Lemma~\ref{lemma:e_pq} to conclude that 
\begin{equation}\label{eq:suff_2}
        \E_{p,q} \mu (b_1, \dots , b_{2k})_i = \sum_{(a_1, \dots , a_{2k}) \in \{0, 1, ?\}^{2k} } f_1(a_1, \dots , a_{2k}, b_1, \dots , b_{2k}) \mu(a_1, \dots , a_{2k})_{i+1} 
    \end{equation}
    where
    \begin{align*}
    f_1(a_1, \dots , a_{2k}, b_1, \dots , b_{2k}) = \prod_{t = 1}^k \widehat{\varphi}_{p,q}(a_{2t - 1}, a_{2t}, b_{2t-1}) \widehat{\varphi}_{p,q}(a_{2t-1}, a_{2t}, b_{2t}).
    \end{align*}
    From the fact that $\mu$ satisfies the property (ii) above, it follows that $\mu(a_1, \dots, a_{2k})_{i+1} = \mu(a_{2k}, \dots, a_1)_{i+1}$. We note that 
    
    \begin{align*}
        f_1(a_1, \dots, a_{2k}, b_1, \dots, b_{2k}) &= \prod_{t = 1}^k \widehat{\varphi}_{p,q}(a_{2t - 1}, a_{2t}, b_{2t-1}) \widehat{\varphi}_{p,q}(a_{2t-1}, a_{2t}, b_{2t}) \\
        &= \prod_{s = 1}^k \widehat{\varphi}_{p,q}(a_{2k - 2s + 1}, a_{2k- 2s + 2}, b_{2k - 2s + 1}) \widehat{\varphi}_{p,q}(a_{2k - 2s + 1}, a_{2k - 2s + 2}, b_{2k - 2s +2}) \\
        &= f_1( a_{2k}, \dots, a_1, b_{2k}, \dots b_1 )
    \end{align*}
    This allows us to re-write equation \ref{eq:suff_2} as follows:
    \begin{align*}
        \E_{p,q} \mu (b_1, \dots , b_{2k})_i &= \sum_{(a_1, \dots, a_{2k}) \in \{0, 1, ?\}^{2k} } f_1(a_1, \dots, a_{2k}, b_1, \dots , b_{2k}) \mu(a_1, \dots, a_{2k})_{i+1} \\
        &= \sum_{(a_{2k}, \dots, a_1) \in \{0, 1, ?\}^{2k} } f_1(a_{2k}, \dots, a_1 , b_{2k}, \dots, b_1 ) \mu(a_{2k}, \dots, a_1)_{i+1} \\
        &= \E_{p, q} \mu (b_{2k}, \dots, b_1)_{ i } 
    \end{align*}
    
    as required. In the last step, we have used Lemma~\ref{lemma:e_pq}. The other case where $i$ is even can be dealt analogously. 
    
    \item The proof that $\E_{p, q} \mu$ satisfies the third property is analogous to proof of the previous part.
    
    \item In this part, we prove that $\E_{p, q} \mu$ satisfies last property. We begin with a cylinder set $( b_1 b_2 \dots b_{2k} )_i$ with $i$ being an odd integer. Then using Lemma~\ref{lemma:e_pq}, we see that 
    \begin{equation}\label{eq:suff_4}
        \E_{p,q} \mu (b_1, \dots , b_{2k})_i = \sum_{(a_1, \dots , a_{2k}) \in \{0, 1, ?\}^{2k} } f_1(a_1, \dots , a_{2k}, b_1, \dots , b_{2k}) \mu(a_1, \dots , a_{2k})_{i+1} 
    \end{equation}
    where
    \begin{align*}
    f_1(a_1, \dots , a_{2k}, b_1, \dots , b_{2k}) = \prod_{t = 1}^k \widehat{\varphi}_{p,q}(a_{2t - 1}, a_{2t}, b_{2t-1}) \widehat{\varphi}_{p,q}(a_{2t-1}, a_{2t}, b_{2t}).
    \end{align*}
    Let $j \in \{1, \dots, 2k\}$ be an odd integer. Using the fact that $i+j$ is an even number,it follows easily from the definition of $f_1$ that 
    \begin{align*}
        f_1(a_1, \dots, a_{2k}, b_1, \dots, b_{2k}) = f_1(a_1, \dots, a_{2k}, b_1, \dots, b_{j-1}, b_{j+1}, b_j, b_{j+1}, \dots, b_{2k})
    \end{align*}

    This allows us to re-write equation~\ref{eq:suff_4} as 
    \begin{align*}
        \E_{p,q} \mu (b_1, \dots , b_{2k})_i &= \sum_{(a_1, \dots , a_{2k}) \in \{0, 1, ?\}^{2k} } f_1(a_1, \dots , a_{2k}, b_1, \dots , b_{2k}) \mu(a_1, \dots , a_{2k})_{i+1} \\
        &= \sum_{(a_1, \dots , a_{2k}) \in \{0, 1, ?\}^{2k} } f_1(a_1, \dots , a_{2k}, b_1, \dots, b_{j-1}, b_{j+1}, b_j, b_{j+1}, \dots, b_{2k}) \mu(a_1, \dots , a_{2k})_{i+1}\\
        &= \E_{p,q}\mu(b_1, \dots, b_{j-1}, b_{j+1}, b_j, b_{j+1}, \dots, b_{2k})_{i}
    \end{align*}
    as required. We note here that the last step follows from Lemma~\ref{lemma:e_pq}. Other cases concerning the parities of $i$ and $k$ follow analogously.

\end{enumerate}
This completes the proof of the lemma.

\end{proof}

\begin{remark} Via an argument identical to that outlined at the very beginning of the proof of Proposition 2.3 of \cite{holroyd2019percolation}, and using Lemmas~\ref{lemma_basic}, \ref{lemma:suffice}, we conclude the following; to prove Theorem~\ref{thm:main_theorem}, it suffices to show that $\mu(?)_i = 0$ for each $i \in \{0, 1\}$ for every stationary distribution $\mu$ of $\E_{p, q}$ that satisfies the properties (i)-(iv). Consequently, we confine ourselves to such distributions. In fact it is easy to see that such distributions $\mu$ satisfy $\mu(?)_0 = \mu(?)_1$.

\end{remark}

For the sake of brevity, we let $\widehat{**} = \{0, ?\}^2 \setminus \{(0, 0)\}$ and $\widehat{1*} = \{0, 1, ?\}^2 \setminus \{0, ?\}^2$. For some symbols $a_1, \dots, a_k, b_1, \dots, b_l \in \{0, 1, ?\}$ and a natural number $i$, we define $(a_1, \dots, a_k, \widehat{**}, b_1, \dots, b_l)_i$) to be the cylinder set containing configurations $\eta$ satisfying 
\begin{enumerate}
    \item $\eta(j) = a_{j - i +1}$ for $j \leq i + k -1$,
    \item $(\eta(i+k), \eta(i+k+1)) \in \widehat{**}$,
    \item $\eta(j) = a_{j-i+1}$ for $i + k + 2\leq j\leq i+l-1$.
\end{enumerate}

In a similar fashion, we define $(a_1, \dots, a_k, \widehat{1*}, b_1, \dots, b_l)_i$) to be the cylinder set containing configurations $\eta$ satisfying 
\begin{enumerate}
    \item $\eta(j) = a_{j - i +1}$ for $j \leq i + k -1$,
    \item $(\eta(i+k), \eta(i+k+1)) \in \widehat{1*}$,
    \item $\eta(j) = a_{j-i+1}$ for $i + k + 2\leq j\leq i+l-1$.
\end{enumerate}

\subsection{Construction of weight functions}We begin by constructing the weight function for the first part of Theorem~\ref{thm:main_theorem}. In what follows, given a cylinder set $(a_1, \dots, a_k)_i$ indexed at $i$, we will use Lemma~\ref{lemma:e_pq} repeatedly to compute $\E\mu(a_1, \dots, a_k)_i$. As in \cite{bhasin2022class} we will begin with a suitable initial weight function and then update it accordingly to obtain an inequality of the desired form. We begin with the initial weight function:

\begin{equation}\label{w_0}
    w_0(\mu) = 2\mu(?0)_1 + \mu(??)_1
\end{equation}
Using Lemma~\ref{lemma:e_pq},
\begin{align}
    w_0(\E_p\mu) &= 2p(1-p) \mu(\widehat{**})_0 + (1-p)^2 \mu(\widehat{**})_0 \nonumber \\
    &= (1-p^2) \mu(\widehat{**}) \label{eq:Ep_w_0}
\end{align}
Re-writing $w_0(\mu)$ by re-arranging some terms, we obtain:
\begin{align}
    w_0(\mu) &= 2(\mu(0?0)_0 + \mu(??0)_0 + \mu(1?0)_0) + \mu(0??)_0 + \mu(???)_0 + \mu(1??)_0 \nonumber \\
    &= 2(\mu(0?0)_0 + \mu(0?1)_0 + \mu(0??)_0) + \mu(??0)_0 + \mu(??1)_0 +\mu(???)_0 + (\mu(1?0)_0 \nonumber \\ & \enspace+ \mu(??0)_0 + \mu(1?0)_0 + \mu(1??)_0) - \mu(0??)_0 \nonumber\\
    &= (2\mu(0?)_0 + \mu(??)_0) + (\mu(1?0)_0 + \mu(??0)_0 + \mu(1?0)_0 + \mu(1??)_0) - \mu(0??)_0 \nonumber\\
    &= \mu(\widehat{**})_0 + (2\mu(1?0)_0 + \mu(??0)_0 + \mu(1??)_0) - \mu(0??)_0 \label{w_0_diff}
\end{align}

Subtracting equation~\ref{eq:Ep_w_0} from both sides:
\begin{align*}
    w_0(\mu) - w_0(\E_p\mu) &= \mu(\widehat{**})_0 + \mu(??0)_0 + \mu(1?0)_0 + \mu(1??)_0) - \mu(0??)_0 - (1-p^2) \mu(\widehat{**})\\
    &= p^2 \mu(\widehat{**})_0 + (2\mu(1?0)_0 + \mu(??0)_0 + \mu(1??)_0) - \mu(0??)_0
\end{align*}

The negative term $-\mu(0??)_0$ suggests that we should update our weight function as follows:
\begin{equation}\label{w_1}
    w_1(\mu) = w_0(\mu) + \mu(0??)_0
\end{equation}
Using Lemma~\ref{lemma:e_pq}, we obtain:
\begin{align*}
    \E_p\mu(0??)_0 &= p(1-p)^2 (\mu(00\widehat{**})_0 + \mu(\widehat{**}\widehat{**})_0) + (1-p)^2 (\mu(\widehat{1*}\widehat{**})_0)\\
    &= p(1-p)^2 (\mu(00\widehat{**})_0 + \mu(\widehat{**}\widehat{**})_0) + (1-p)^2 (\mu(100?)_0 + \mu(10?0)_0 + \mu(10??)_0 \\&\enspace+ \mu(1?0?)_0 + \mu(1??0)_0 + \mu(1???)_0 + \mu(010?)_0 + \mu(01?0)_0 + \mu(01??)_0 \\&\enspace+ \mu(?10?)_0 + \mu(?1?0)_0 + \mu(?1??)_0 + \mu(110?)_0 + \mu(11?0)_0 + \mu(11??)_0) \\
    &= p(1-p)^2 (\mu(00\widehat{**})_0 + \mu(\widehat{**}\widehat{**})_0) + (1-p)^2 (\mu(100?)_0 + \mu(10?0)_0 + \mu(10??)_0 \\&\enspace+ \mu(1?0?)_0 + \mu(1??0)_0 + \mu(1???)_0 + \mu(010?)_0 + \mu(?10?)_0 +  \mu(110?)_0)  
\end{align*}

Using equations~\ref{w_0_diff} and \ref{w_1}, we obtain: 
\begin{align}
    w_1(\mu) - w_1(\E_p\mu) &= p^2 \mu(\widehat{**})_0 + (2\mu(1?0)_0 + \mu(??0)_0 + \mu(1??)_0) - \mu(0??)_0 + \mu(0??_0) \nonumber \\
    &\enspace -p(1-p)^2 (\mu(00\widehat{**})_0 + \mu(\widehat{**}\widehat{**})_0) - (1-p)^2 (\mu(100?)_0 + \mu(10?0)_0 + \mu(10??)_0 \nonumber \\&\enspace+ \mu(1?0?)_0 + \mu(1??0)_0 + \mu(1???)_0 + \mu(010?)_0 + \mu(?10?)_0 +  \mu(110?)_0) \nonumber \\
    &= p^2 \mu(\widehat{**})_0 + (\mu(1?0)_0 + \mu(??0)_0) + (\mu(1?0)_0 + \mu(1??)_0) \nonumber \\
    &\enspace -p(1-p)^2 (\mu(00\widehat{**})_0 + \mu(\widehat{**}\widehat{**})_0) - (1-p)^2 (\mu(100?)_0 + \mu(10?0)_0 + \mu(10??)_0 \nonumber  \\&\enspace+ \mu(1?0?)_0 + \mu(1??0)_0 + \mu(1???)_0 + \mu(010?)_0 + \mu(?10?)_0 +  \mu(110?)_0) \nonumber \\
    &= p^2 \mu(\widehat{**})_0 + (\mu(1?0)_0 + \mu(??0)_0) + (\mu(1?00)_0 + \mu(1?01)_0 +\mu(1?0?)_0 + \mu(1??0)_0 \nonumber \\
    &\enspace+ \mu(1??1)_0 + \mu(1???)_0)  -p(1-p)^2 (\mu(00\widehat{**})_0 + \mu(\widehat{**}\widehat{**})_0) - (1-p)^2 (\mu(100?)_0 \nonumber \\
    &\enspace+ \mu(10?0)_0 + \mu(10??)_0 + \mu(1?0?)_0 + \mu(1??0)_0 + \mu(1???)_0 + \mu(010?)_0 \nonumber \\
    &\enspace+ \mu(?10?)_0 +  \mu(110?)_0)\nonumber \\
    &= p^2 \mu(\widehat{**})_0 + (\mu(1?0)_0 + \mu(??0)_0) + (\mu(1?01)_0) + \mu(1??1)_0 + (\mu(10?0)_0 +\mu(1?0?)_0 \nonumber \\
    &\enspace+ \mu(1??0)_0 + \mu(1???)_0  -p(1-p)^2 (\mu(00\widehat{**})_0 + \mu(\widehat{**}\widehat{**})_0) - (1-p)^2 (\mu(100?)_0 \nonumber \\
    &\enspace+ \mu(10?0)_0 + \mu(10??)_0 + \mu(1?0?)_0 + \mu(1??0)_0 + \mu(1???)_0 + \mu(010?)_0 \nonumber \\
    &\enspace+ \mu(?10?)_0 +  \mu(110?)_0) \nonumber \\
    &= p^2 \mu(\widehat{**})_0 + (\mu(1?0)_0 + \mu(??0)_0) + (\mu(1?01)_0 + \mu(1??1)_0) + p (2-p) (\mu(10?0)_0 \label{w_1_diff} \\
    &\enspace +\mu(1?0?)_0 + \mu(1??0)_0 + \mu(1???)_0 -p(1-p)^2 (\mu(00\widehat{**})_0 + \mu(\widehat{**}\widehat{**})_0) \nonumber\\
    &\enspace- (1-p)^2 (\mu(100?)_0 + \mu(10??)_0 + \mu(010?)_0 + \mu(?10?)_0 +  \mu(110?)_0)\nonumber
\end{align}

The terms $-p(1-p)^2 (\mu(00\widehat{**})_0 + \mu(\widehat{**}\widehat{**})_0)$ along with the terms $(\mu(1?0)_0 + \mu(??0)_0)$ suggest we should update our weight function as follows:
\begin{equation}\label{w_2}
    w_2(\mu) = w_1(\mu) - \mu(1?0)_0 - \mu(??0)_0
\end{equation}

This is because of the following equations obtained using Lemma~\ref{lemma:e_pq}:
 \begin{equation}\label{1?0}
        \E_p\mu(1?0)_0 = p(1-p)^2 \mu(00\widehat{**})_0 
    \end{equation}
    
     \begin{equation}\label{??0}
        \E_p\mu(??0)_0 = p(1-p)^2 \mu(\widehat{**}\widehat{**})_0
    \end{equation}
Using equations~\ref{w_1_diff}, \ref{w_2}, \ref{1?0} and \ref{??0}, we obtain that
\begin{align}\label{0.5}
    w_2(\mu) - w_2(\E_p\mu) &= p^2 \mu(\widehat{**})_0 +  (\mu(1?01)_0 + \mu(1??1)_0) + p (2-p) (\mu(10?0)_0 +\mu(1?0?)_0  \nonumber\\ 
    &\enspace+ \mu(1??0)_0 + \mu(1???)_0) - (1-p)^2 (\mu(100?)_0 + \mu(10??)_0 + \mu(10?)_1  )
\end{align}
Before updating our weight function further, we note that the equation \ref{0.5} yields to us the weight function for the case $q=0, p > 0.5$. In particular, we note that 
\begin{equation}\label{0?}
    \mu(0?)_0 \geq \mu(100?)_0 + \mu(10?)_1  
\end{equation}
and that 
\begin{equation}\label{??}
    \mu(??)_0 \geq \mu(10??)_0 
\end{equation}
From equations \ref{0?} and \ref{??} we obtain that 
\begin{equation}
    \mu(\widehat{**})_0 \geq \mu(100?)_0 + \mu(10??)_0 + \mu(10?)_1 
\end{equation}
Using this in  equation\ref{0.5} we obtain that
\begin{align}
    w_2(\mu) - w_2(\E_p\mu) &= p^2 (\mu(\widehat{**})_0 - (\mu(100?)_0 + \mu(10??)_0 + \mu(10?)_1) +  (\mu(1?01)_0 + \mu(1??1)_0) \nonumber\\
    &\enspace+ p (2-p) (\mu(10?0)_0 +\mu(1?0?)_0 + \mu(1??0)_0 + \mu(1???)_0) - (1-p)^2 (\mu(100?)_0 \nonumber\\
    &\enspace+ \mu(10??)_0 + \mu(10?)_1  ) + p^2 (\mu(100?)_0 + \mu(10??)_0 + \mu(10?)_1 )\nonumber\\
    &=  p^2 (\mu(\widehat{**})_0 - (\mu(100?)_0 + \mu(10??)_0 + \mu(10?)_1) +  (\mu(1?01)_0 + \mu(1??1)_0) \label{w_3_diff} \\
    &\enspace+ p (2-p) (\mu(10?0)_0 +\mu(1?0?)_0  + \mu(1??0)_0 + \mu(1???)_0) \\
    &\enspace+ (2p-1) (\mu(100?)_0 + \mu(10??)_0 + \mu(10?)_1  )
\end{align}
All the terms of right hand side of of the above equation are positive when $p>0.5$. Thus, when $\mu$ is a stationary distribution, we obtain that 
\begin{equation}
    \mu(100?)_0 + \mu(10??)_0 + \mu(10?)_1 = 0
\end{equation}
and that 
\begin{equation}
    \mu(\widehat{**})_0 - (\mu(100?)_0 + \mu(10??)_0 + \mu(10?)_1) = 0 
\end{equation}
From these two equations, it follows that $\mu(\widehat{**})_0 = 0$. Finally, from $\mu(?) = (1-p) \mu(\widehat{**})_0$, it follows that $\mu(?) = 0$, as required.

Continuing from where we left, equation~\ref{0.5} suggests us that we should update our weight function as follows:
\begin{equation}\label{w_4}
    w_3(\mu) = w_2(\mu) + \mu(100?)_0 + \mu(10??)_0 + \mu(10?)_1
\end{equation}

Using Lemma~\ref{lemma:e_pq}, we obtain:
\begin{equation}\label{100?}
    \E_p\mu(100?)_0 = p^2(1-p)^2 (\mu(0000\widehat{**})_0 + \mu(00\widehat{**}\widehat{**})_0 ) + (1-p)^2 (\mu(00\widehat{1*}\widehat{**})_0) \nonumber
\end{equation}

\begin{equation}\label{10??}
    \E_p\mu(10??)_0 = p(1-p)^3 (\mu(00\widehat{**}\widehat{**})_0)\nonumber
\end{equation}

\begin{equation}\label{10?}
    \E_p\mu(10?)_1 = p(1-p)^2 (\mu(00\widehat{**})_0) \nonumber
\end{equation}

From equations~\ref{w_3_diff}, \ref{w_4}, \ref{100?}, \ref{10??} and \ref{10?}, it follows that
\begin{align}
    w_3(\mu) - w_3(\E_p\mu) &= p^2 \mu(\widehat{**})_0 +  (\mu(1?01)_0 + \mu(1??1)_0) + p (2-p) (\mu(10?0)_0 +\mu(1?0?)_0  \nonumber\\
    &\enspace+ \mu(1??0)_0 + \mu(1???)_0) - (1-p)^2 (\mu(100?)_0 + \mu(10??)_0 + \mu(10?)_1  )\nonumber \\&\enspace +(\mu(100?)_0 + \mu(10??)_0 + \mu(10?)_1) - ( p^2(1-p)^2 (\mu(0000\widehat{**})_0 \nonumber\\
    &\enspace+ \mu(00\widehat{**}\widehat{**})_0 ) + (1-p)^2 (\mu(00\widehat{1*}\widehat{**})_0) ) -  p(1-p)^3 (\mu(00\widehat{**}\widehat{**})_0)\nonumber \\
    &\enspace- p(1-p)^2 (\mu(00\widehat{**})_0) \nonumber\\
    &= p^2 \mu(\widehat{**})_0 +  (\mu(1?01)_0 + \mu(1??1)_0) + p (2-p) (\mu(10?0)_0 +\mu(1?0?)_0  \nonumber\\
    &\enspace+ \mu(1??0)_0 + \mu(1???)_0 + \mu(100?)_0 + \mu(10??)_0 + \mu(10?)_1 ) \nonumber \\
    &\enspace- ( p^2(1-p)^2 (\mu(0000\widehat{**})_0 + \mu(00\widehat{**}\widehat{**})_0 ) + (1-p)^2 (\mu(00\widehat{1*}\widehat{**})_0) )\nonumber\\
    &\enspace -  p(1-p)^3 (\mu(00\widehat{**}\widehat{**})_0) - p(1-p)^2 (\mu(00\widehat{**}00)_0 + \mu(00\widehat{**}\widehat{**})_0 + \mu(00\widehat{**}\widehat{1*})_0)\nonumber \\
    &= p^2 \mu(\widehat{**})_0 +  (\mu(1?01)_0 + \mu(1??1)_0) + p (2-p) (\mu(10?0)_0 +\mu(1?0?)_0 \nonumber \\
    &\enspace+ \mu(1??0)_0 + \mu(1???)_0 + \mu(100?)_0 + \mu(10??)_0 + \mu(010?)_0 +\mu(110?)_0 \nonumber\\
    &\enspace + \mu(?10?)_0 )  -  p^2(1-p)^2 \mu(0000\widehat{**})_0 -  2p(1-p)^2 \mu(00\widehat{**}\widehat{**})_0 \nonumber\\
    &\enspace- p(1-p)^2 \mu(00\widehat{**}00)_0 - p(1-p)^2 \mu(00\widehat{**}\widehat{1*})_0 - (1-p)^2 \mu(00\widehat{1*}\widehat{**})_0 \nonumber\\
    &= p^2 \mu(\widehat{**})_0 +  (\mu(1?01)_0 + \mu(1??1)_0) + p (2-p) ( \mu(\widehat{1*}\widehat{**}) )  -  p^2(1-p)^2 \mu(0000\widehat{**})_0\nonumber \\
    &\enspace-  2p(1-p)^2 \mu(00\widehat{**}\widehat{**})_0 - p(1-p)^2 \mu(00\widehat{**}00)_0 - p(1-p)^2 \mu(00\widehat{**}\widehat{1*})_0 \nonumber\\
    &\enspace- (1-p)^2 \mu(00\widehat{1*}\widehat{**})_0  \nonumber\\
    &\geq p^2 \mu(\widehat{**})_0 +  (\mu(1?01)_0 + \mu(1??1)_0) + p (2-p) ( \mu(00\widehat{1*}\widehat{**})_0 + \mu(1\widehat{1*}\widehat{**})_1 \nonumber\\
    &\enspace+ \mu(?\widehat{1*}\widehat{**})_1 )   -  p^2(1-p)^2 \mu(0000\widehat{**})_0 -  2p(1-p)^2 \mu(00\widehat{**}\widehat{**})_0 - p(1-p)^2 \mu(00\widehat{**}00)_0 \nonumber\\
    &\enspace- p(1-p)^2 \mu(00\widehat{**}\widehat{1*})_0 - (1-p)^2 \mu(00\widehat{1*}\widehat{**})_0 \nonumber\\
    &= p^2 \mu(\widehat{**})_0 +  (\mu(1?01)_0 + \mu(1??1)_0) + p (2-p) ( \mu(1\widehat{1*}\widehat{**})_1 + \mu(?\widehat{1*}\widehat{**})_1 )  \label{w_4_diff}\\
    &\enspace -  p^2(1-p)^2 \mu(0000\widehat{**})_0 -  2p(1-p)^2 \mu(00\widehat{**}\widehat{**})_0 - p(1-p)^2 \mu(00\widehat{**}00)_0\nonumber \\
    &\enspace- p(1-p)^2 \mu(00\widehat{**}\widehat{1*})_0 - (1+2p^2 -4p) \mu(00\widehat{1*}\widehat{**})_0 \nonumber
\end{align}

We update the weight function as follows:

\begin{equation}\label{w_5}
    w_4(\mu) = w_3(\mu) - p(2-p) \mu(1\widehat{1*}\widehat{**})_1 - p(2-p) \mu(?\widehat{1*}\widehat{**})_1 
\end{equation}

Since $\E_p \mu(1?)_1 = 0$, we can simplify equation~\ref{w_5} as follows:
\begin{align*}
    \mu(1\widehat{1*}\widehat{**})_1 &= \mu(1100?)_1 + \mu(110?0)_1 + \mu(110??)_1 + \mu(11?0?)_1 + \mu(11??0)_1 + \mu(11???)_1 \\&\enspace+ \mu(1110?)_1 + \mu(111?0)_1 + \mu(111??)_1 + \mu(1010?)_1 + \mu(101?0)_1 + \mu(101??)_1 \\&\enspace+ \mu(1?10?)_1 + \mu(1?1?0)_1 + \mu(1?1??)_1 \\
    &= \mu(1100?)_1 + \mu(110?0)_1 + \mu(110??)_1 + \mu(11?0?)_1 + \mu(11??0)_1 + \mu(11???)_1 \\&\enspace+ \mu(1110?)_1 + \mu(1010?)_1 
\end{align*}
and
\begin{align*}
    \mu(?\widehat{1*}\widehat{**})_1 &= \mu(?100?)_1 + \mu(?10?0)_1 + \mu(?10??)_1 + \mu(?1?0?)_1 + \mu(?1??0)_1 + \mu(?1???)_1 \\&\enspace+ \mu(?110?)_1 + \mu(?11?0)_1 + \mu(?11??)_1 + \mu(?010?)_1 + \mu(?01?0)_1 + \mu(?01??)_1 \\&\enspace+ \mu(??10?)_1 + \mu(??1?0)_1 + \mu(??1??)_1 \\
    &= \mu(?010?)_1 + \mu(??10?)_1 
\end{align*}

These two identities allow us to re-write equation~\ref{w_5} as follows:
\begin{align}
    w_4(\mu) &= w_3(\mu) - p(2-p)( \mu(1100?)_1 + \mu(110?0)_1 + \mu(110??)_1 + \mu(11?0?)_1 + \mu(11??0)_1 \nonumber\\&\enspace+ \mu(11???)_1 + \mu(1110?)_1 + \mu(1010?)_1 +  \mu(?010?)_1 + \mu(??10?)_1  ) \label{w_5'} 
\end{align}

Using Lemma~\ref{lemma:e_pq}, we obtain the following:
\begin{equation}
    \E_p\mu(1100?)_1 = p^2 (1-p)^3 (\mu(0000\widehat{**})_0 + \mu(00\widehat{**}\widehat{**})_0 ) + (1-p)^3 \mu(00\widehat{1*}\widehat{**})_0 \nonumber
\end{equation}

\begin{equation}
   \E_p\mu(110?0)_1 = p^2 (1-p)^3 (\mu(00\widehat{**}00)_0 + \mu(00\widehat{**}\widehat{**})_0 ) + p(1-p)^3 \mu(00\widehat{**}\widehat{1*})_0 \nonumber
\end{equation}

\begin{equation}
    \E_p\mu(110??)_1 = p(1-p)^4 \mu(00\widehat{**}\widehat{**})_0 \nonumber
\end{equation}

\begin{equation}
    \E_p\mu(11?0?)_1 = p(1-p)^4 \mu(00\widehat{**}\widehat{**})_0 \nonumber
\end{equation}

\begin{equation}
    \E_p\mu(11??0)_1 = p(1-p)^4 (\mu(00\widehat{**}00)_0 + \mu(00\widehat{**}\widehat{**})_0) + (1-p)^4\mu(00\widehat{**}\widehat{1*}) \nonumber
\end{equation}

\begin{equation}
    \E_p\mu(11???)_1 = (1-p)^5 \mu(00\widehat{**}\widehat{**})_0 \nonumber
\end{equation}

\begin{equation}
    \E_p\mu(1110?)_1 = p(1-p)^4 \mu(0000\widehat{**})_0 \nonumber
\end{equation}

\begin{equation}
    \E_p\mu(1010?)_1 = p^2(1-p)^3 \mu(0000\widehat{**})_0 \nonumber
\end{equation}

\begin{equation}
    \E_p\mu(?010?)_1 = p^2(1-p)^3 \mu(\widehat{**}00\widehat{**})_0 \nonumber
\end{equation}

\begin{equation}
    \E_p\mu(??10?)_1 = p(1-p)^4 \mu(\widehat{**}00\widehat{**})_0 \nonumber
\end{equation}

Adding the previous ten equations and gathering coefficients of various cylinder sets, we obtain: 
\begin{align}
    w_4(\E_p\mu) \leq& w_3(\E_p\mu) - p(2-p) (p(1-p)^3(1+p) \mu(0000\widehat{**})_0 + (1-p)^3(1+p) \mu(00\widehat{**}\widehat{**})_0 \nonumber \\&+ (1-p)^3 \mu(00\widehat{**}\widehat{1*})_0 + (1-p)^3\mu(00\widehat{1*}\widehat{**})_0 ) \label{w_5_E}
\end{align}

Using \ref{w_4_diff}, \ref{w_5'} and \ref{w_5_E}, we obtain the following:
\begin{align*}
    w_4(\mu) - w_4(\E_p\mu) &\geq p^2 \mu(\widehat{**})_0 +  (\mu(1?01)_0 + \mu(1??1)_0) + p (2-p) ( \mu(1\widehat{1*}\widehat{**})_1 + \mu(?\widehat{1*}\widehat{**})_1 )  \\
    &\enspace -  p^2(1-p)^2 \mu(0000\widehat{**})_0 -  2p(1-p)^2 \mu(00\widehat{**}\widehat{**})_0 - p(1-p)^2 \mu(00\widehat{**}00)_0 \\
    &\enspace- p(1-p)^2 \mu(00\widehat{**}\widehat{1*})_0 - (1+2p^2 -4p) \mu(00\widehat{1*}\widehat{**})_0 - p(2-p) (\mu(1\widehat{1*}\widehat{**})_1\\
    &\enspace + \mu(?\widehat{1*}\widehat{**})_1) + p(2-p) (p(1-p)^3(1+p) \mu(0000\widehat{**})_0 + (1-p)^3(1+p) \mu(00\widehat{**}\widehat{**})_0 \\&+ (1-p)^3 \mu(00\widehat{**}\widehat{1*})_0 + (1-p)^3\mu(00\widehat{1*}\widehat{**})_0 )\\
    &= p^2 \mu(\widehat{**})_0 +  (\mu(1?01)_0 + \mu(1??1)_0)    -  p^2(1-p)^2 \mu(0000\widehat{**})_0   \\
    &\enspace - 2p(1-p)^2 \mu(00\widehat{**}\widehat{**})_0 - p(1-p)^2 \mu(00\widehat{**}00)_0 - p(1-p)^2 \mu(00\widehat{**}\widehat{1*})_0  \\
    &\enspace- (1+2p^2 -4p) \mu(00\widehat{1*}\widehat{**})_0  + p(2-p) (p(1-p)^3(1+p) \mu(0000\widehat{**})_0  \\
    &\enspace+ (1-p)^3(1+p) \mu(00\widehat{**}\widehat{**})_0 + (1-p)^3 \mu(00\widehat{**}\widehat{1*})_0 + (1-p)^3\mu(00\widehat{1*}\widehat{**})_0 )
\end{align*}

Finally, we update the weight function as follows:
\begin{equation}\label{w_6}
    w_5(\mu) = w_4(\mu) - \mu(1?01)_0 - \mu(1??1)_0 
\end{equation}

Using Lemma~\ref{lemma:e_pq}, we obtain:
\begin{equation}\label{1?01}
    \E_p\mu(1?01)_0 = p(1-p)^3 \mu(00\widehat{**}00)_0 \nonumber
\end{equation}

\begin{equation}\label{1??1}
    \E_p\mu(1??1)_0 = (1-p)^4 \mu(00\widehat{**}00)_0 \nonumber
\end{equation}

From \ref{w_6}, \ref{1?01} and \ref{1??1}, we see that:
\begin{align*}
    w_5(\mu) - w_5(\E_p\mu) &\geq  p^2 \mu(\widehat{**})_0 +  (\mu(1?01)_0 + \mu(1??1)_0)    -  p^2(1-p)^2 \mu(0000\widehat{**})_0   \\
    &\enspace - 2p(1-p)^2 \mu(00\widehat{**}\widehat{**})_0 - p(1-p)^2 \mu(00\widehat{**}00)_0 - p(1-p)^2 \mu(00\widehat{**}\widehat{1*})_0  \\
    &\enspace- (1+2p^2 -4p) \mu(00\widehat{1*}\widehat{**})_0  + p(2-p) (p(1-p)^3(1+p) \mu(0000\widehat{**})_0  \\
    &\enspace+ (1-p)^3(1+p) \mu(00\widehat{**}\widehat{**})_0 + (1-p)^3 \mu(00\widehat{**}\widehat{1*})_0 + (1-p)^3\mu(00\widehat{1*}\widehat{**})_0 ) \\
    &\enspace- (\mu(1?01)_0 + \mu(1??1)_0) +  p(1-p)^3 \mu(00\widehat{**}00)_0 + (1-p)^4 \mu(00\widehat{**}00)_0\\
    &= p^2 \mu(\widehat{**})_0     -  p^2(1-p)^2 \mu(0000\widehat{**})_0  - 2p(1-p)^2 \mu(00\widehat{**}\widehat{**})_0  \\
    &\enspace -p(1-p)^2 \mu(00\widehat{**}00)_0 - p(1-p)^2 \mu(00\widehat{**}\widehat{1*})_0  - (1+2p^2 -4p) \mu(00\widehat{1*}\widehat{**})_0    \\
    &\enspace + p(2-p) (p(1-p)^3(1+p) \mu(0000\widehat{**})_0  + (1-p)^3(1+p) \mu(00\widehat{**}\widehat{**})_0  \\
    &\enspace+ (1-p)^3 \mu(00\widehat{**}\widehat{1*})_0 + (1-p)^3\mu(00\widehat{1*}\widehat{**})_0 )  +  p(1-p)^3 \mu(00\widehat{**}00)_0  \\
    &\enspace+ (1-p)^4 \mu(00\widehat{**}00)_0
\end{align*}

And we use the term $p^2 \mu(\widehat{**})_0$ as follows:

\begin{center}
\begin{tabular}{ |c|c|c|c| } 
 \hline
  &  & $\widehat{**}$ &  \\ 
 \hline
  & 00 & $\widehat{**}$ & $\widehat{**}$ \\
  \hline 
   & 00 & $\widehat{**}$ & 00 \\
  \hline
   & 00  &  $\widehat{**}$& $\widehat{1*}$ \\
  \hline
   00 & $\widehat{1*}$ & $\widehat{**}$ &  \\
  \hline
  
\end{tabular}
\end{center}

From the above table, we can conclude that:
\begin{equation}
    \mu(\widehat{**})_0 \geq \mu(00\widehat{**}\widehat{**})_0 + \mu(00\widehat{**}00)_0 + \mu(00\widehat{**}\widehat{1*})_0 + \mu(00\widehat{1*} \widehat{**})_0
\end{equation}

\begin{align}
    w_5(\mu) - w_5(\E_p\mu) &\geq   p^2 \mu(\widehat{**})_0     -  p^2(1-p)^2 \mu(0000\widehat{**})_0  - 2p(1-p)^2 \mu(00\widehat{**}\widehat{**})_0 \nonumber \\
    &\enspace -p(1-p)^2 \mu(00\widehat{**}00)_0 - p(1-p)^2 \mu(00\widehat{**}\widehat{1*})_0  - (1+2p^2 -4p) \mu(00\widehat{1*}\widehat{**})_0    \nonumber\\
    &\enspace + p(2-p) (p(1-p)^3(1+p) \mu(0000\widehat{**})_0  + (1-p)^3(1+p) \mu(00\widehat{**}\widehat{**})_0 \nonumber \\
    &\enspace+ (1-p)^3 \mu(00\widehat{**}\widehat{1*})_0 + (1-p)^3\mu(00\widehat{1*}\widehat{**})_0 )  +  p(1-p)^3 \mu(00\widehat{**}00)_0 \nonumber \\
    &\enspace+ (1-p)^4 \mu(00\widehat{**}00)_0 \nonumber\\
    &= p^2(\mu(\widehat{**}_0 - (\mu(00\widehat{**}\widehat{**})_0 + \mu(00\widehat{**}00)_0 + \mu(00\widehat{**}\widehat{1*})_0 + \mu(00\widehat{1*} \widehat{**})_0)) \nonumber \\
    &\enspace+ p^2(\mu(00\widehat{**}\widehat{**})_0 + \mu(00\widehat{**}00)_0 + \mu(00\widehat{**}\widehat{1*})_0 + \mu(00\widehat{1*} \widehat{**})_0) \nonumber \\
    &\enspace- p^2(1-p)^2 \mu(0000\widehat{**})_0  - 2p(1-p)^2 \mu(00\widehat{**}\widehat{**})_0  -p(1-p)^2 \mu(00\widehat{**}00)_0 \nonumber \\
    &\enspace- p(1-p)^2 \mu(00\widehat{**}\widehat{1*})_0  - (1+2p^2 -4p) \mu(00\widehat{1*}\widehat{**})_0  \nonumber  \\
    &\enspace + p(2-p) (p(1-p)^3(1+p) \mu(0000\widehat{**})_0  + (1-p)^3(1+p) \mu(00\widehat{**}\widehat{**})_0 \nonumber \\
    &\enspace+ (1-p)^3 \mu(00\widehat{**}\widehat{1*})_0 + (1-p)^3\mu(00\widehat{1*}\widehat{**})_0 )  +  p(1-p)^3 \mu(00\widehat{**}00)_0 \nonumber \\
    &\enspace+ (1-p)^4 \mu(00\widehat{**}00)_0 \nonumber \\
    &= p^2(\mu(\widehat{**})_0 - (\mu(00\widehat{**}\widehat{**})_0 + \mu(00\widehat{**}00)_0 + \mu(00\widehat{**}\widehat{1*})_0 + \mu(00\widehat{1*} \widehat{**})_0)) \label{coeff_poly} \\
    &\enspace+ p^2(1-p)^2(p^3-2p^2-p+1) \mu(0000\widehat{**}) + p^4(p-2)^2 \mu(00\widehat{**}\widehat{**})_0 \nonumber \\
    &\enspace+ (1-4p + 6p^2 -2p^3) \mu(00\widehat{**}00)_0 + p(p^4 - 5p^3 + 8p^2 -4p + 1) \mu(00\widehat{**}\widehat{1*})_0 \nonumber \\
    &\enspace+ (p^5 - 5p^4 + 9p^3 - 8p^2 + 6p -1)\mu(00\widehat{1*}\widehat{**})_0 \nonumber
\end{align}

Note that the polynomial $p^5 - 5p^4 + 9p^3 - 8p^2 + 6p -1$ assumes positive values for $p>p_0$. This polynomial, alongwith some other polynomials shows up in \ref{coeff_poly} where in order to obtain our result, we want all those polynomials to take positive values. It turns out that all the polynomials are positive in $ ( p_0, p_1) $ where $p_1\approx 0.555$ and is the second largest root of $p^3-2p^2-p+1$.

Using Wolfram alpha we see that the coefficients of all the cylinder sets above are positive when $p\in  ( p_0, p_1) $ where $p_0 \approx 0.215$ is the unique real root of $p^5 - 5p^4 + 9p^3 - 8p^2 + 6p -1$ and $p_1\approx 0.555$ and is the second largest root of $p^3-2p^2-p+1$. For $p \in (p_0, p_1)$ and $\mu$ stationary, we obtain that
\begin{enumerate}
    \item $\mu(0000\widehat{**})_0 = 0$.
    \item $\mu(00\widehat{**}\widehat{**})_) = 0$.
    \item $\mu(00\widehat{**}00)_0 = 0$.
    \item $\mu(00\widehat{**}\widehat{1*})_0 = 0$.
    \item $\mu(00\widehat{1*}\widehat{**}) = 0$.
    \item $\mu(\widehat{**})_0 = \mu(00\widehat{**}\widehat{**})_0 + \mu(00\widehat{**}00)_0 + \mu(00\widehat{**}\widehat{1*})_0 + \mu(00\widehat{1*} \widehat{**})_0 = 0 $
\end{enumerate}

Finally, from $\mu(\widehat{**})_0 = 0$ and $E_p\mu(?)=(1-p)\mu(\widehat{**})_0$, it follows that $\mu(?) = 0$.

\vspace{1mm}
 For the second part of Theorem~\ref{thm:main_theorem}, we begin with the initial weight function 
 \begin{align}
     w_0(\mu) &= \mu(??)_1 + \mu(1?)_1 + \mu(?1)_1 + \mu(0?)_1 + \mu(?0)_1 \nonumber\\
     &= \mu(??)_1 + 2\mu(?1)_1 + 2\mu(?0)_1 \nonumber\\
     &= \mu(0??)_0 + \mu(1??)_0 + \mu(???)_0 + 2 ( \mu(0?1)_0 + \mu(1?1)_0 + \mu(??1)_0) + 2 ( \mu(0?0)_0 \nonumber\\ 
     & \enspace + \mu(1?0)_0 + \mu(??0)_0) \label{gw_0} 
 \end{align}
 
Using Lemma~\ref{lemma:e_pq}, we obtain the following identities:
\begin{equation}\label{g??}
    \E_{p, q} \mu(??)_1 = r^2 \mu(\widehat{**})_0
\end{equation}
\begin{equation}\label{g?1}
    \E_{p, q} \mu(?1)_1 = rq \mu(\widehat{**})_0
\end{equation}
\begin{equation}\label{g?0}
    \E_{p, q} \mu(?0)_1 = rp \mu(\widehat{**})_0
\end{equation}

Adding equations~\ref{g??}, \ref{g?1} and \ref{g?0}, we obtain that:
\begin{align}
    w_0(\E_{p, q} \mu) &= (r^2 + 2rp + 2rq ) \mu(\widehat{**})_0 \nonumber\\
    &= (1 - (p+q)^2 ) \mu(\widehat{**})_0\nonumber \\
    &= (1 - (p+q)^2 ) ( \mu(0?)_0 + \mu(?0)_0 + \mu(??)_0) \nonumber\\
    &= (1 - (p+q)^2 ) ( 2\mu(0?)_0  + \mu(??)_0)\nonumber \\
    &= (1 - (p+q)^2 ) ( 2 (\mu(0?0)_0 + \mu(1?0)_0 + \mu(??0)_0) + \mu(??0)_0 + \mu(??1)_0 + \mu(???)_0\label{gw_0_E}
\end{align}

using equations~\ref{gw_0} and \ref{gw_0_E}, we obtain:

\begin{align*}
    w_0(\mu) - w_0(\E_{p,q}\mu) &= (p+q)^2 \mu(\widehat{**})_0 + \mu(1??)_0 + \mu(??1)_0 + 2\mu(1?1)_0  + 2\mu(1?0)_0 \\ 
    &\enspace+ \mu(??0)_0 - \mu(0?0)_0 \\
    &= (p+q)^2 \mu(\widehat{**})_0 + \mu(01??)_1 + \mu(11??)_1 + \mu(?1??)_1 + \mu(??1)_0 \\ 
    &\enspace+ 2\mu(1?1)_0 + 2\mu(1?0)_0 + \mu(??0)_0 - \mu(00??)_1 - \mu(10??)_1 - \mu(?0??)_1\\
    &= (p+q)^2 \mu(\widehat{**})_0 + \mu(11??)_1 + \mu(?1??)_1 + \mu(??1)_0 + 2\mu(1?1)_0 \\ 
    &\enspace+ 2\mu(1?0)_0 + \mu(??0)_0 - \mu(00??)_1 - \mu(?0??)_1
\end{align*}
 
 We update our weight function as follows:
 
 \begin{align*}
     w_1(\mu) &= w_0(\mu) - \mu(11??)_1 - \mu(?1??)_1 - \mu(??1)_0 - 2\mu(1?1)_0 \\ 
    &\enspace - 2\mu(1?0)_0 - \mu(??0)_0 + \mu(00??)_1 + \mu(?0??)_1
 \end{align*}
 
Using Lemma~\ref{lemma:e_pq} we have the following identities:

\begin{equation}
    \E_{p, q}\mu(11??)_1 = (1-p)^2 r^2 \mu(00\widehat{**})_0 + q^2 r^2 \mu(\widehat{1*}\widehat{**})_0 + q^2 r^2 \mu(\widehat{**}\widehat{**})_0
\end{equation}

\begin{equation}
    \E_{p,q}\mu(?1??)_1 = q r^3 \mu(\widehat{**}\widehat{**})_0
\end{equation}

\begin{equation}
    \E_{p, q}\mu(??1)_0 = q r^2 \mu(\widehat{**}\widehat{**})_0
\end{equation}

\begin{equation}
    \E_{p,q}\mu(1?1)_0 = (1-p)^2 r \mu(00\widehat{**})_0 + r q^2 \mu(\widehat{1*}\widehat{**})_0 + r q^2 \mu(\widehat{**}\widehat{**})_0
\end{equation}

\begin{equation}
    \E_{p,q}\mu(1?0)_0 = (1-p)pr \mu(00\widehat{**})_0 + pqr \mu(\widehat{1*}\widehat{**})_0 pqr \mu(\widehat{**}\widehat{**})_0
\end{equation}

\begin{equation}
    \E_{p,q}\mu(??0)_0 = p r^2  \mu(\widehat{**}\widehat{**})_0
\end{equation}

\begin{equation}
    \E_{p,q}\mu(00??)_1 = p^2 r^2 \mu(00\widehat{**})_0 + (1-q)^2 r^2 \mu(\widehat{1*}\widehat{**})_0 + p^2 r^2 \mu(\widehat{**}\widehat{**})_0
\end{equation}

Using these equations and rewriting $ \mu(\widehat{**})_0 $ as $ \mu(00\widehat{**})_0 + \mu(\widehat{1*}\widehat{**})_0 + \mu(\widehat{**}\widehat{**})_0 $, we obtain the following:

\begin{align*}
    w_1(\mu) - w_1(\E_{p,q}\mu) &= ( (p+q)^2 + (1-p)^2 r^2 + (1-p)pr + (1-p)rq -p^2r^2 )\mu(00\widehat{**})_0 \\
    &\enspace+ ((p+q)^2 + q^2 r^2 + 2pqr  + 2rq^2 - (1-q)^2r^2) \mu(\widehat{1*}\widehat{**})_0 \\
    &\enspace+ ( (p+q)^2 + q^2 r^2 + qr^3 + 2pqr + r^2p + 2rq^2 + r^2q - p^2r^2 - pr^3 ) \mu(\widehat{**}\widehat{**})_0\\
    &= ( (p+q)^2 + (1-p)^2 r^2 + (1-p)pr + (1-p)rq -p^2r^2 )\mu(00\widehat{**})_0 \\
    &\enspace+ (-2q^2 + 2q(2-p) + 2p -1) \mu(\widehat{1*}\widehat{**})_0 \\
    &\enspace+ ( (p+q)^2 + q^2 r^2 + qr^3 + 2pqr + r^2p + 2rq^2 + r^2q - p^2r^2 - pr^3 ) \mu(\widehat{**}\widehat{**})_0\\
    &\geq ( (1-p)^2 r^2 + (1-p)pr -p^2r^2 )\mu(00\widehat{**})_0 \\
    &\enspace+ (-2q^2 + 2q(2-p) + 2p -1) \mu(\widehat{1*}\widehat{**})_0 \\
    &\enspace+ ( r^2p - p^2r^2 - pr^3 ) \mu(\widehat{**}\widehat{**})_0\\
    &= (r^2(1-2p) + rp(1-p)) \mu(00\widehat{**})_0 + (-2q^2 + 2q(2-p) + 2p -1) \mu(\widehat{1*}\widehat{**})_0 \\
    &\enspace + r^2 p q \mu(\widehat{**}\widehat{**})_0\\
    &\geq (-r^2p + rp(1-p)) \mu(00\widehat{**})_0 + (-2q^2 + 2q(2-p) + 2p -1) \mu(\widehat{1*}\widehat{**})_0  \\
    &\enspace + r^2 p q \mu(\widehat{**}\widehat{**})_0\\
     &\geq rpq \mu(00\widehat{**})_0 + (-2q^2 + 2q(2-p) + 2p -1) \mu(\widehat{1*}\widehat{**})_0 + r^2 p q \mu(\widehat{**}\widehat{**})_0
\end{align*}

Using the quadratic formula, it is easy to see that the polynomial is positive in the following range:
\begin{enumerate}
    \item $0\leq p<\frac{1}{2}$, $q > \frac{-\sqrt{p^2 +2}-p+2}{2}$.
    \item $\frac{1}{2} \leq p <1$, $0 < q < 1$.
\end{enumerate}

When $\mu$ is a stationary distribution and $p$ and $q$ are in the above mentioned range, we can conclude that
\begin{enumerate}
    \item $ \mu(00\widehat{**})_0 = 0$,
    \item $\mu(\widehat{1*}\widehat{**})_0 = 0$,
    \item $\mu(\widehat{**}\widehat{**})_0 = 0$.
\end{enumerate}
This allows us to conclude that $\mu(\widehat{**})_0 = 0$. Finally, using $\E_{p, q}\mu(?) = r \mu(\widehat{**})_0$, we see that $\E_{p, q}\mu(?) = 0$. Since $\mu$ is stationary, we obtain the desired result of Theorem~\ref{thm:main_theorem}.

\bibliography{mainTCS}
\nocite{*}

\end{document}